\documentclass[oneside,english]{amsart}
\usepackage[T1]{fontenc}
\usepackage[latin9]{inputenc}
\usepackage{amsthm}
\usepackage{amstext}
\usepackage{amssymb}
\usepackage{esint}
\usepackage{graphicx}
\usepackage{enumerate}
\usepackage{amscd}

\makeatletter
\newtheorem{theorem}{Theorem}[section]
\newtheorem{lemma}[theorem]{Lemma}
\newtheorem{proposition}[theorem]{Proposition}

\numberwithin{figure}{section}
\theoremstyle{definition}

\newtheorem{example}[theorem]{Example}

\theoremstyle{remark}
\newtheorem{remark}[theorem]{Remark}

\numberwithin{equation}{section}
\makeatother

\usepackage{color}

\usepackage{babel}

	\DeclareMathOperator{\dist}{dist}
	\DeclareMathOperator{\loc}{loc}

    \DeclareMathOperator{\Imag}{Im}
    \DeclareMathOperator{\Real}{Re}

\begin{document}

\title[Dirichlet eigenvalues problem]{On the principal frequency of non-homogeneous membranes}

\author{V.~Gol'dshtein, V.~Pchelintsev}
\begin{abstract}
We obtained estimates for first eigenvalues of the Dirichlet boundary value problem for  elliptic operators in divergence form (i.e. for the principal frequency of non-homogeneous membranes) in bounded domains $\Omega \subset \mathbb C$ satisfying quasihyperbolic boundary conditions. The suggested method is based on the quasiconformal composition operators on Sobolev spaces and their applications to constant estimates in the corresponding Sobolev-Poincar\'e inequalities. We also prove a variant of the Rayleigh-Faber-Khran inequality for a special case of these elliptic operators.
\end{abstract}
\maketitle

\section{Introduction}

In this paper we give applications of the theory of quasiconformal mappings to the Dirichlet eigenvalue problem for two-dimensional elliptic operators in divergence form
\begin{equation}\label{EllDivOper}
L_{A}f(z)=-\textrm{div} [A(z) \nabla f(z)], \quad z=(x,y)\in \Omega, \quad f(x,y)=0\,\,\text{on}\,\,\partial\Omega,
\end{equation}
in bounded domains $\Omega \subset \mathbb C$ satisfying quasihyperbolic boundary conditions \cite{KOT01,KOT02}.
We assume that $A \in M^{2 \times 2}(\Omega)$, where $M^{2 \times 2}(\Omega)$ is the class of all $2 \times 2$ symmetric matrix functions $A(z)=\left\{a_{kl}(z)\right\}$, $\textrm{det} A=1$ a.~e. in $\Omega$, with measurable entries satisfying to the uniform ellipticity condition
\begin{equation}\label{UnEllCon}
\frac{1}{K}|\xi|^2 \leq \left\langle A(w) \xi, \xi \right\rangle \leq K |\xi|^2 \,\,\, \text{a.~e. in}\,\,\, \Omega,
\end{equation}
for every $\xi \in \mathbb C$, where $1\leq K< \infty$. Such elliptic operators in divergence form arise in various problems of mathematical physics (see, for example, \cite{AIM}).

Under these conditions the matrix $A$ induces a quasiconformal homeomorphism $\varphi_A:\Omega\to\widetilde{\Omega}$
which we call the $A$-quasiconformal mapping \cite{GPU2020}. This allow us to reduce (by quasiconformal change of variable) this elliptic operator in divergence form to the Laplace operator.

A domain $\Omega$ satisfies the $\gamma$-quasihyperbolic boundary condition with some $\gamma>0$ if the growth condition on the quasihyperbolic metric
$$
k_{\Omega}(x_0,x)\leq \frac{1}{\gamma}\log\frac{\dist(x_0,\partial\Omega)}{\dist(x,\partial\Omega)}+C_0
$$
satisfied for all $x\in\Omega$, where $x_0\in\Omega$ is a fixed base point and $C_0=C_0(x_0)<\infty$,
\cite{GM,H1}. Here
$$
k_{\Omega}(x_1,x_2):= \inf\int\limits_{\gamma} \frac{ds}{\dist(x,\partial\Omega)},
$$
where the infimum is taken over all rectifiable curves $\gamma$ joining $x_1$ and $x_2$.

This class of domains includes, in particular, domains with Lipschitz boundary, some domains with H\"older singularities, and domains of snowflakes type \cite{R01}.

A function $f\in W_0^{1,2}(\Omega,A)$ is a solution to the generalized spectral problem for the elliptic
operator in divergence form $L_{A}f(z)$ with the Dirichlet boundary condition if
\[
\iint\limits_{\Omega}\langle A(z)\nabla f(z),\nabla \overline{g(z)}\rangle~dxdy
= \lambda
\iint\limits_{\Omega}f(z)\overline{g(z)}~dxdy,\,\, \forall g\in W^{1,2}_{0}(\Omega,A).
\]

It is known \cite{Henr,M} that in a bounded domain $\Omega \subset \mathbb C$ the operator $L_{A}f(z)$ with the Dirichlet boundary
condition has discrete spectrum represented as the non-decreasing sequence
\[
0< \lambda_1(A,\Omega) \leq \lambda_2(A,\Omega) \leq \ldots \leq \lambda_n(A,\Omega) \leq \ldots,
\]
where each eigenvalue is repeated as many time as its multiplicity.
By the min-max principle, the first
eigenvalue $\lambda_1(A,\Omega)$ is defined by
\[
\lambda_1(A,\Omega)=\inf_{f \in W_0^{1,2}(\Omega,A) \setminus \{0\}}
\frac{\|f\mid L^{1,2}_{A}(\Omega)\|^2}{\|f\mid L^2(\Omega)\|^2}\,.
\]

The lower estimates for the first eigenvalues of the Laplace operator with the Dirichlet boundary condition in a bounded domain are connected by the Rayleigh-Faber-Khran inequality \cite{F23,Kh25} which means that the first Dirichlet eigenvalue
in a bounded domain $\Omega$ is not less than the corresponding Dirichlet eigenvalue in the disc of the same area
$\Omega^{\ast}$ with $R_{\ast}$ as its radius, i.e.,
\begin{equation*}
\lambda_1(I,\Omega):=\lambda_1(\Omega)\geq \lambda_1(\Omega^{\ast})=\frac{{j_{0,1}^2}}{R^2_{\ast}},
\end{equation*}
where $j_{0,1} \approx 2.4048$ is the first positive zero of the Bessel function $J_0$. This inequality was improved by the method based on the capacity theory \cite{M}.

Unfortunately, for the first eigenvalue $\lambda_1(A,\Omega)$ Rayleigh-Faber-Khran type inequality has not been proven. However, lower estimates for the first eigenvalues the operator $L_{A}f(z)$ with the Dirichlet boundary
condition in bounded domains can be obtained easily by using
the Rayleigh-Faber-Krahn inequality and the uniform ellipticity condition~\eqref{UnEllCon}:
\textit{Let $\Omega \subset \mathbb C$ be a bounded domain such that $|\Omega|=|\Omega^{\ast}|$ and $K$ is the ellipticity constant of the matrix $A$.
Then
\begin{equation}\label{est-4.2}
\lambda_1(A,\Omega) \geq \frac{\lambda_1(\Omega)}{K} \geq \frac{\lambda_1(\Omega^{\ast})}{K}=\frac{{j_{0,1}^2}}{KR^2_{\ast}}.
\end{equation} }

In this paper we obtain lower estimates for the first Dirichlet eigenvalues of the divergence form elliptic operators $L_{A}f(z)$ in bounded  domains with some quasihyperbolic boundary condition.  We call such domains as $\beta$-regular domains for some $\beta \in (1, \infty]$. The class of all $\beta$-regular domains coincides with the class of all domains with quasihyperbolic boundary conditions.

 Our machinery is based on connections between $A$-quasiconformal
mappings \cite{AIM,BGMR,GNR18} and composition operators on Sobolev spaces \cite{GPU2020}.

One of the main results of the article states the following estimate for $\infty$-regular domains: \textit{
If a simply connected bounded  domain $\Omega \subset \mathbb C$ satisfies to the quasihyperbolic boundary condition, then
\begin{equation}\label{est-4.1}
\lambda_1(A,\Omega) \geq
\frac{\lambda_1(\widetilde{\Omega})}{\|J_{\varphi_{A}^{-1}}\mid L^{\infty}(\widetilde{\Omega})\|},
\end{equation}
where $\lambda_1(\widetilde{\Omega})$ is the first Dirichlet eigenvalue of the Laplace operator and $J_{\varphi_{A}^{-1}}$ is  a Jacobian of the inverse mapping
to the $A$-quasiconformal mapping $\varphi_{A}:\Omega \to \widetilde{\Omega}$.}

\vskip 0.2cm

A detailed discussion about $\beta$-regular domains can be found in Section 3.

Note that if $\widetilde{\Omega}=\Omega^{\ast}$ and $\|J_{\varphi_{A}^{-1}}\mid L^{\infty}(\Omega^{\ast})\|<K$ then estimate~\eqref{est-4.1} is better than
estimates~\eqref{est-4.2}. For example, this condition is satisfied for measure preserving $A$-quasiconformal mappings $\varphi_{A}:\Omega\to\Omega^{\ast}$ ($|J(z,\varphi_A)|=1$ a.e. in $\Omega$). Some examples can be found at the end of this paper.

Taking into account the domain monotonicity property for the Dirichlet eigenvalues of the operator $L_{A}f(z)$ (see, for example, \cite{Henr}) and estimate~\eqref{est-4.1} we obtain estimates for variations of the first Dirichlet eigenvalues of the operator $L_{A}f(z)$ under quasiconformal deformations of the domain. Namely:
\textit{Let $\Omega \subset \mathbb C$ be a bounded $\infty$-regular domain. We assume that $\varphi_A(\Omega):=\widetilde{\Omega} \supset \Omega$, then
\[
\lambda_1(A,\Omega)-\lambda_1(\widetilde{\Omega}) \geq \frac{1-\|J_{\varphi_{A}^{-1}}\mid L^{\infty}(\widetilde{\Omega})\|}{ \|J_{\varphi_{A}^{-1}}\mid L^{\infty}(\widetilde{\Omega})\|}\lambda_1(\widetilde{\Omega}),
\]
where $\lambda_1(\widetilde{\Omega})$ is the first Dirichlet eigenvalue of the Laplace operator and $J_{\varphi_{A}^{-1}}$ is  a Jacobian of the inverse mapping
to the $A$-quasiconformal mapping $\varphi_{A}:\Omega \to \widetilde{\Omega}$.}

\vskip 0.2cm

In the case of the measure preserving $A$-quasiconformal mappings $\varphi_{A}:\Omega\to\widetilde{\Omega}$ we prove Rayleigh-Faber-Khran type inequality for the operator $L_{A}f(z)$:
\textit{Let $\Omega \subset \mathbb C$ be a simply connected bounded domain such that there exists a measure preserving A-quasiconformal mapping $\varphi_{A}:\Omega \to \widetilde{\Omega}$. $|\Omega|=|\Omega^{\ast}|$. Then
\[
\lambda_1(A, \Omega) \geq \lambda_1(\Omega^{\ast})=\frac{{j_{0,1}^2}}{R^2_{\ast}}.
\]
Here $\Omega^{\ast}$ is the disc of the radius $R_{\ast}$ such that $|\Omega|=|\Omega^{\ast}|$. and $j_{0,1} \approx 2.4048$ is the first positive zero of the Bessel function $J_0$.}

\section{Sobolev spaces and $A$-quasiconformal mappings}

Let $E \subset \mathbb C$ be a measurable set on the complex plane and $h:E \to \mathbb R$ be a positive almost everywhere (a.e.) locally integrable function, i.e. a weight. The weighted Lebesgue space $L^p(E,h)$, $1\leq p<\infty$,
is the space of all locally integrable functions endowed with the following norm
$$
\|f\,|\,L^{p}(E,h)\|= \left(\iint\limits_E|f(z)|^ph(z)\,dxdy \right)^{\frac{1}{p}}< \infty.
$$

The two-weighted Sobolev space $W^{1,p}(\Omega,h,1)$, $1\leq p< \infty$, is defined
as the normed space of all locally integrable weakly differentiable functions
$f:\Omega\to\mathbb{R}$ endowed with the following norm:
\[
\|f\mid W^{1,p}(\Omega,h,1)\|=\|f\,|\,L^{p}(\Omega,h)\|+\|\nabla f\mid L^{p}(\Omega)\|.
\]

In the case $h=1$ this weighted Sobolev space coincides with the classical Sobolev space $W^{1,p}(\Omega)$.
The seminormed Sobolev space $L^{1,p}(\Omega)$, $1\leq p< \infty$,
is the space of all locally integrable weakly differentiable functions $f:\Omega\to\mathbb{R}$ endowed
with the following seminorm:
\[
\|f\mid L^{1,p}(\Omega)\|=\|\nabla f\mid L^p(\Omega)\|, \,\, 1\leq p<\infty.
\]

We also need a weighted seminormed Sobolev space $L^{1,2}_{A}(\Omega)$ (associated with the matrix $A$), defined
as the space of all locally integrable weakly differentiable functions $f:\Omega\to\mathbb{R}$
with the finite seminorm given by:
\[
\|f\mid L^{1,2}_{A}(\Omega)\|=\left(\iint\limits_\Omega \left\langle A(z)\nabla f(z),\nabla f(z)\right\rangle\,dxdy \right)^{\frac{1}{2}}.
\]

The corresponding  Sobolev space $W^{1,2}(\Omega, A)$ is defined
as the normed space of all locally integrable weakly differentiable functions
$f:\Omega\to\mathbb{R}$ endowed with the following norm:
\[
\|f\mid W^{1,2}(\Omega, A)\|=\|f\,|\,L^{2}(\Omega)\|+\|f\mid L^{1,2}_{A}(\Omega)\|.
\]
The Sobolev space $W^{1,2}_{0}(\Omega, A)$ is the closure in the $W^{1,2}(\Omega, A)$-norm of the
space $C^{\infty}_{0}(\Omega)$.

We consider the Sobolev spaces as Banach spaces of equivalence classes of functions up to a set of $p$-capacity zero \cite{M}.

Recall that a homeomorphism $\varphi: \Omega \to \widetilde{\Omega}$, $\Omega,\, \widetilde{\Omega} \subset\mathbb C$, is called a $Q$-quasiconformal mapping if $\varphi\in W^{1,2}_{\loc}({\Omega})$ and there exists a constant $1\leq Q<\infty$ such that
$$
|D\varphi(z)|^2\leq Q |J(z,\varphi)|\,\,\text{for almost all}\,\,z \in \Omega.
$$

Note that quasiconformal mappings have a finite distortion and possesses the Luzin $N$-property (i.e. a image of any set of measure zero has measure zero) \cite{VGR}.

If $\varphi : \Omega \to \widetilde{\Omega}$ is a $Q$-quasiconformal mapping then $\varphi$ is differentiable almost everywhere in $\Omega$ and
$$
|J(z,\varphi)|=J_{\varphi}(z):=\lim\limits_{r\to 0}\frac{|\varphi(B(z,r))|}{|B(z,r)|}\,\,\text{for almost all}\,\,z\in\Omega.
$$

Now we give a construction of $A$-quasiconformal mappings connected with the matrix $A$.

Recall that matrix functions $A(z)=\left\{a_{kl}(z)\right\}$  with measurable entries $a_{kl}(z)$ belongs to a class  $M^{2 \times 2}(\Omega)$ of all $2 \times 2$ symmetric matrix functions that satisfy to an additional condition $\textrm{det} A=1$ a.e. and to the uniform ellipticity condition:
\begin{equation}\label{UEC}
\frac{1}{K}|\xi|^2 \leq \left\langle A(z) \xi, \xi \right\rangle \leq K |\xi|^2 \,\,\, \text{a.e. in}\,\,\, \Omega,
\end{equation}
for every $\xi \in \mathbb C$ and for some $1\leq K< \infty$.
The basic idea is that every positive quadratic form
\[
ds^2=a_{11}(x,y)dx^2+2a_{12}(x,y)dxdy+a_{22}(x,y)dy^2
\]
defined in a planar domain $\Omega$ can be reduced, by means of a quasiconformal change of variables, to the canonical form
\[
ds^2=\Lambda(du^2+dv^2),\,\, \Lambda\neq 0,\,\, \text{a.e. in}\,\, \widetilde{\Omega},
\]
given that $a_{11}a_{22}-a^2_{12} \geq \kappa_0>0$, $a_{11}>0$, almost everywhere in $\Omega$ \cite{Ahl66,BGMR}.

By \cite{BGMR} any matrix $A$ of the type under discussion induces a quasiconformal homeomorphism as a solution to the corresponding Beltrami equation. The detailed procedure is described below

Let $\xi(z)=\Real \varphi(z)$ be a real part of a quasiconformal mapping $\varphi(z)=\xi(z)+i \eta(z)$, which satisfies to the Beltrami equation:
\begin{equation}\label{BelEq}
\varphi_{\overline{z}}(z)=\mu(z) \varphi_{z}(z),\,\,\, \text{a.e. in}\,\,\, \Omega,
\end{equation}
where
$$
\varphi_{z}=\frac{1}{2}\left(\frac{\partial \varphi}{\partial x}-i\frac{\partial \varphi}{\partial y}\right) \quad \text{and} \quad
\varphi_{\overline{z}}=\frac{1}{2}\left(\frac{\partial \varphi}{\partial x}+i\frac{\partial \varphi}{\partial y}\right),
$$
with the complex dilatation $\mu(z)$ given by
\begin{equation}\label{ComDil}
\mu(z)=\frac{a_{22}(z)-a_{11}(z)-2ia_{12}(z)}{\det(I+A(z))},\quad I= \begin{pmatrix} 1 & 0 \\ 0 & 1 \end{pmatrix}.
\end{equation}
We call this quasiconformal mapping (with the complex dilatation $\mu$ defined by (\ref{ComDil})) as an $A$-quasiconformal mapping and we will use the notation $\varphi_A$  for this quasiconformal mapping.

Note that the uniform ellipticity condition \eqref{UEC} can be written as
\begin{equation}\label{OVCE}
|\mu(z)|\leq \frac{K-1}{K+1},\,\,\, \text{a.e. in}\,\,\, \Omega.
\end{equation}

Conversely from \eqref{ComDil} (see, for example, \cite{AIM}, p. 412) one can recover the matrix $A$ :
\begin{equation}\label{Matrix-F}
A(z)= \begin{pmatrix} \frac{|1-\mu|^2}{1-|\mu|^2} & \frac{-2 \Imag \mu}{1-|\mu|^2} \\ \frac{-2 \Imag \mu}{1-|\mu|^2} &  \frac{|1+\mu|^2}{1-|\mu|^2} \end{pmatrix},\,\,\, \text{a.e. in}\,\,\, \Omega.
\end{equation}

Thus, for any $A \in M^{2 \times 2}(\Omega)$ by \eqref{OVCE} can be produced the complex dilatation $\mu(z)$, for which, in turn, the Beltrami equation \eqref{BelEq} induces an $A$-quasiconformal homeomorphism $\varphi:\Omega \to \widetilde{\Omega}$ as its solution (by the Riemann measurable mapping theorem (see, for example, \cite{Ahl66})). Let us briefly say that $A$ and $\varphi_A$ are agreed.

Therefore with the given $A$-divergent form elliptic operator defined in a domain $\Omega\subset\mathbb C$ can be assochiated the $A$-quasiconformal mapping $\varphi_A:\Omega \to \widetilde{\Omega}$ with the quasiconformality coefficient
$$
Q_A=\frac{1+\|\mu\mid L^{\infty}(\Omega)\|}{1-\|\mu\mid L^{\infty}(\Omega)\|},
$$
where $\mu$ defined by (\ref{ComDil}).

From the estimate $|\mu(z)|\leq \frac{K-1}{K+1}$ immediately follows that $Q_A \leq K$.

Note that the inverse mapping to the $A$-quasiconformal mapping $\varphi_A: \Omega \to \widetilde{\Omega}$ is the $A^{-1}$-quasiconformal mapping \cite{GPU2020}.

In \cite{GPU2020} the relationship between composition operators on Sobolev spaces and $A$-quasiconformal mappings was studied and the following theorem was proved.

\begin{theorem} \label{Isometry}
Let $\Omega,\widetilde{\Omega}$ be domains in $\mathbb C$. Then a homeomorphism $\varphi_A :\Omega \to \widetilde{\Omega}$ is an $A$-quasiconformal mapping
if and only if $\varphi$ induces, by the composition rule $\varphi^{*}(f)=f \circ \varphi$,
an isometry of Sobolev spaces $L^{1,2}_A(\Omega)$ and $L^{1,2}(\widetilde{\Omega})$ i.e.
\[
\|\varphi_A^{*}(f)\,|\,L^{1,2}_A(\Omega)\|=\|f\,|\,L^{1,2}(\widetilde{\Omega})\|
\]
for any $f \in L^{1,2}(\widetilde{\Omega})$.
\end{theorem}

This theorem generalizes the well known property of conformal mappings generate the isometry of uniform Sobolev spaces $L^1_2(\Omega)$ and $L^1_2(\widetilde{\Omega})$ (see, for example, \cite{C50}).  It is also refines (in the case $n=2$) the functional characterization of quasiconformal mappings in the terms of isomorphisms of uniform Sobolev spaces \cite{VG75}.

\section{Estimate of the constant in Sobolev-Poincar\'e inequality}

In \cite{GPU2020_2} was proved
the following weighted Sobolev-Poincar\'e inequality for a bounded domain ${\Omega}\subset\mathbb C$.
We denote by $h(z) = |J(z,\varphi_A)|$ the quasihyperbolic weight defined by an A-quasiconformal
mapping $\varphi_A:\Omega \to \widetilde{\Omega}$.

\begin{theorem}\label{Th3.1}
Let $A$ belongs to a class  $M^{2 \times 2}(\Omega)$ and $\Omega$ be a bounded simply connected planar domain.
Then for any function $f \in W^{1,2}_{0}(\Omega,A)$ the following weighted Sobolev-Poincar\'e inequality
\[
\left(\iint\limits_\Omega |f(z)|^rh(z)dxdy\right)^{\frac{1}{r}} \leq C_{r,2}(h,A,\Omega)
\left(\iint\limits_\Omega \left\langle A(z) \nabla f(z), \nabla f(z) \right\rangle dxdy\right)^{\frac{1}{2}}
\]
holds for any $r \geq 2$ with the constant $C_{r,2}(h,A,\Omega) = C_{r,2}(\widetilde{\Omega})$.
\end{theorem}
Here $C_{r,2}(\widetilde{\Omega})$ is the best constant in the (non-weight) Sobolev-Poincar\'e inequality in a bounded domain
$\widetilde{\Omega}\subset\mathbb C$ with the upper estimate (see \cite{GPU2019}):
\begin{equation}\label{Const}
C_{r,2}(\widetilde{\Omega}) \leq \inf\limits_{p\in \left(\frac{2r}{r+2},2\right)}
\left(\frac{p-1}{2-p}\right)^{\frac{p-1}{p}}
\frac{\left(\sqrt{\pi}\cdot\sqrt[p]{2}\right)^{-1}|\widetilde{\Omega}|^{\frac{1}{r}}}{\sqrt{\Gamma(2/p) \Gamma(3-2/p)}}.
\end{equation}

Using Theorem~\ref{Th3.1} we give an upper estimate for the constant in the Sobolev-Poincar\'e inequality in domains with the
quasihyperbolic boundary condition.  As shown in \cite{AK}, the Jacobians of quasiconformal mappings $\psi:\widetilde{\Omega}\to\Omega$ belong to $L^{\beta}(\widetilde{\Omega})$ for some $\beta >1$ if and only if ${\Omega}$ has the $\gamma$-quasihyperbolic boundary condition for some $\gamma$. We note that $\beta$ depends only on $\widetilde{\Omega}$ and the quasiconformal coefficient $K(\psi)$.

Since we need the exact value of the integrability exponent $\beta$ for the Jacobians of quasiconformal mappings, we consider an equivalent description of domains with quasihyperbolic boundary in terms of the integrability of Jacobians \cite{GPU19}.

A simply connected domain $\Omega \subset \mathbb C$ is called an $A$-quasiconformal $\beta$-regular domain about a simply connected domain $\widetilde{\Omega} \subset \mathbb C$ if
\[
\iint\limits_{\widetilde{\Omega}} |J(w,\varphi_{A}^{-1})|^{\beta}~dudv<\infty
\]
for some $\beta>1$, where $\varphi_{A}:\Omega \to \widetilde{\Omega}$ is a corresponding $A$-quasiconformal mapping.

The property of the quasiconformal $\beta$-regularity implies the integrability of
a Jacobian of quasiconformal mappings and therefore for any quasiconformal $\beta$-regular domain we have the embedding of weighted Lebesgue spaces $L^r(\Omega,h)$ into
non-weighted Lebesgue spaces $L^s(\Omega)$ for $s=\frac{\beta -1}{\beta}r$ \cite{GPU19}.
\begin{lemma}\label{Prop-reg} \cite{GPU19}
Let $\Omega$ be an $A$-quasiconformal $\beta$-regular domain. Then for any
function $f \in L^r(\Omega,h)$, $\beta/(\beta -1) \leq r < \infty$, the inequality
\[
\|f\mid L^s(\Omega)\|\leq \left(\iint\limits_{\widetilde{\Omega}}|J(w,\varphi^{-1}_A)|^{\beta}dudv\right)^{\frac{1}{\beta} \cdot \frac{1}{s}}
\|f\mid L^r(\Omega,h)\|
\]
holds for $s=\frac{\beta -1}{\beta}r$.
\end{lemma}

We ready to prove the upper estimate for the constant in the Sobolev-Poincar\'e inequality in quasiconformal regular domains.

\begin{theorem}\label{Th4.3}
Let $A$ belong to a class  $M^{2 \times 2}(\Omega)$ and a domain $\Omega$ be $A$-quasi\-conformal $\beta$-regular about $\widetilde{\Omega}$. Then:
\begin{enumerate}
\item for any function $f \in W^{1,2}_{0}(\Omega,A)$ and for any $s \geq 1$, the Sobolev-Poincar\'e inequality
\[
\|f\mid L^s(\Omega)\| \leq C_{s,2}(A,\Omega)
\|f\mid L^{1,2}_{A}(\Omega)\|
\]
holds with the constant
$$
C_{s,2}(A,\Omega) \leq C_{\frac{\beta s}{\beta-1},2}(\widetilde{\Omega}) \|J_{\varphi_{A}^{-1}}\mid L^{\beta}(\widetilde{\Omega})\|^{\frac{1}{s}}, \quad 1<\beta <\infty;
$$
\item if $\beta = \infty$ then for any function $f \in W^{1,2}_{0}(\Omega, A)$, the Sobolev-Poincar\'e inequality
\[
\|f\mid L^2(\Omega)\| \leq C_{2,2}(A,\Omega)
\|f\mid L^{1,2}_{A}(\Omega)\|
\]
holds with the constant
$$C_{2,2}(A,\Omega) \leq C_{2,2}(\widetilde{\Omega}) \big\|J_{\varphi_{A}^{-1}}\mid L^{\infty}(\widetilde{\Omega})\big\|^{\frac{1}{2}}\,.
$$
Here $J_{\varphi_{A}^{-1}}$ is  a Jacobian of the inverse mapping
to the $A$-quasiconformal mapping $\varphi_{A}:\Omega \to \widetilde{\Omega}$.
\end{enumerate}
\end{theorem}

\begin{remark}
The constant $C_{2,2}^2(\widetilde{\Omega})=1/\lambda_1(\widetilde{\Omega})$, where $\lambda_1(\widetilde{\Omega})$ is the first Dirichlet eigenvalue of Laplacian in a domain $\widetilde{\Omega}\subset\mathbb C$.
\end{remark}

\begin{proof}
1. Let $\frac{\beta}{\beta-1}\leq r$ and $s=\frac{\beta -1}{\beta}r$. It means that such $r$ exists for any $s \geq 1$. By Lemma~\ref{Prop-reg} and Theorem~\ref{Th3.1} we get
\begin{multline*}
\|f\mid L^s(\Omega)\|=\left(\iint\limits_{\Omega}|f(z)|^sdxdy\right)^\frac{1}{s} \\
\leq
\left(\iint\limits_{\widetilde{\Omega}}|J(w,\varphi^{-1}_A)|^{\beta}dudv\right)^{\frac{1}{\beta} \cdot \frac{1}{s}}
\left(\iint\limits_{\Omega}|f(z)|^{r}|J(\varphi_A,x,y)|dxdy\right)^{\frac{1}{r}} \\
\leq
C_{r,2}(\widetilde{\Omega})
\left(\iint\limits_{\widetilde{\Omega}}|J(w,\varphi^{-1}_A)|^{\beta}dudv\right)^{\frac{1}{\beta} \cdot \frac{1}{s}}
\left(\iint\limits_\Omega \left\langle A(z) \nabla f(z), \nabla f(z) \right\rangle dxdy\right)^{\frac{1}{2}}
\end{multline*}
for any $s \geq 1$.

2. Let a function $f\in L^2(\Omega)$. Since quasiconformal mappings possess the Luzin $N$-property, then $|J(z,\varphi_A)|^{-1}=|J(w,\varphi_A^{-1})|$ for almost all $z\in \Omega$ and for almost all $w=\varphi_A(z)\in \widetilde{\Omega}$.
Hence
\begin{multline*}
\left(\iint\limits_{\Omega} |f(z)|^2~dxdy\right)^{\frac{1}{2}}
=\left(\iint\limits_{\Omega} |f(z)|^2|J(z,\varphi_A)|^{-1}|J(z,\varphi_A)|~dxdy\right)^{\frac{1}{2}} \\
\leq \|J_{\varphi_A} \mid L^{\infty}(\Omega)\|^{-\frac{1}{2}} \left(\int\limits_{\Omega} |f(z)|^2|J(z,\varphi_A)|~dxdy\right)^{\frac{1}{2}}.
\end{multline*}
Applying the change of variable formula for quasiconformal mappings \cite{VGR}, (non-weighed) Sobolev-Poincar\'e inequality \cite{GPU2019}, and Theorem~\ref{Isometry} we obtain
\begin{multline*}
\left(\iint\limits_{\Omega} |f(z)|^2~dxdy\right)^{\frac{1}{2}} \\
\leq
\|J_{\varphi_{A}^{-1}} \mid L^{\infty}(\widetilde{\Omega})\|^{\frac{1}{2}}
\left(\iint\limits_{\widetilde{\Omega}} |f \circ \varphi_{A}^{-1}(w)|^2~dudv\right)^{\frac{1}{2}} \\
\leq
C_{2,2}(\widetilde{\Omega})
\|J_{\varphi_{A}^{-1}} \mid L^{\infty}(\widetilde{\Omega})\|^{\frac{1}{2}}
\left(\iint\limits_{\widetilde{\Omega}} |\nabla(f \circ \varphi_{A}^{-1}(w))|^2~dudv\right)^{\frac{1}{2}} \\
=
C_{2,2}(\widetilde{\Omega})
\|J_{\varphi_{A}^{-1}} \mid L^{\infty}(\widetilde{\Omega})\|^{\frac{1}{2}}
\left(\iint\limits_\Omega \left\langle A(z) \nabla f(z), \nabla f(z) \right\rangle dxdy\right)^{\frac{1}{2}}
\end{multline*}
for any $f\in W_0^{1,2}(\Omega,A)$.
\end{proof}

\section{Lower estimates for $\lambda_1(A,\Omega)$}

We consider the generalized formulation of the Dirichlet eigenvalue problem~\eqref{EllDivOper}:
\[
\iint\limits_{\Omega}\langle A(z)\nabla f(z),\nabla \overline{g(z)}\rangle~dxdy
= \lambda
\iint\limits_{\Omega}f(z)\overline{g(z)}~dxdy,\,\, \forall g\in W^{1,2}_{0}(\Omega,A).
\]
By the min-max principle (see, for example, \cite{Henr,M}) the first Dirichlet eigenvalue $\lambda_1(A,\Omega)$ of the elliptic operator in divergence form $L_{A}f(z)$ is defined by
\[
\lambda_1(A,\Omega)=\inf_{f \in W_0^{1,2}(\Omega,A) \setminus \{0\}}
\frac{\|f\mid L^{1,2}_{A}(\Omega)\|^2}{\|f\mid L^2(\Omega)\|^2}\,.
\]
In other words, $\lambda_1^{-\frac{1}{2}}(A,\Omega)$ is the exact constant $C_{2,2}(A,\Omega)$ in the Sobolev-Poincar\'e inequality
\[
\|f\mid L^2(\Omega)\| \leq C_{2,2}(A,\Omega)\|f\mid L^{1,2}_{A}(\Omega)\|,\,\, f\in W^{1,2}_{0}(\Omega,A).
\]

\begin{theorem}
Let $A$ belong to the class  $M^{2 \times 2}(\Omega)$ and $\Omega$ be an $A$-quasiconformal $\beta$-regular domain about a domain $\widetilde{\Omega}$. Then
\[
\frac{1}{\lambda_1(A,\Omega)} \leq C^2_{\frac{2\beta}{\beta -1},2}(\widetilde{\Omega})\|J_{\varphi_{A}^{-1}}\mid L^{\beta}(\widetilde{\Omega})\|,
\]
where $J_{\varphi_{A}^{-1}}$ is  a Jacobian of the inverse mapping
to the $A$-quasiconformal mapping $\varphi_{A}:\Omega \to \widetilde{\Omega}$ and
\[
C_{\frac{2\beta}{\beta -1},2}(\widetilde{\Omega}) \leq \inf\limits_{p\in \left(\frac{2\beta}{2\beta -1},2\right)}
\left(\frac{p-1}{2-p}\right)
\frac{\left(\sqrt{\pi}\cdot\sqrt[p]{2}\right)^{-1}|\widetilde{\Omega}|^{\frac{\beta-1}{2\beta}}}{\sqrt{\Gamma(2/p) \Gamma(3-2/p)}}.
\]
\end{theorem}

\begin{proof}
By the min-max principle and Theorem~\ref{Th4.3} in the case $s=2$, we have
\[
\iint\limits_{\Omega}|f(z)|^2~dxdy \leq C^2_{2,2}(A,\Omega)
\iint\limits_{\Omega} \langle A(z)\nabla f(z),\nabla f(z)\rangle~dxdy,
\]
where
\[
C_{2,2}(A,\Omega)\leq C_{\frac{2\beta}{\beta -1},2}(\widetilde{\Omega})
\left(\iint\limits_{\widetilde{\Omega}}|J(w,\varphi_{A}^{-1})|^{\beta}~dudv\right)^{\frac{1}{2\beta}}.
\]
Thus
\[
\frac{1}{\lambda_1(A,\Omega)} \leq C^2_{\frac{2\beta}{\beta -1},2}(\widetilde{\Omega})
\left(\iint\limits_{\widetilde{\Omega}}|J(w,\varphi_{A}^{-1})|^{\beta}~dudv\right)^{\frac{1}{\beta}}.
\]
\end{proof}

In the limit case $\beta=\infty$ we have the following assertion:
\begin{theorem}\label{Th-LE}
Let $A$ belong to a class  $M^{2 \times 2}(\Omega)$ and $\Omega$ be an $A$-quasiconformal $\infty$-regular domain about a domain $\widetilde{\Omega}$. Then
\begin{equation}\label{Estim}
\frac{1}{\lambda_1(A,\Omega)} \leq C^2_{2,2}(\widetilde{\Omega})\|J_{\varphi_{A}^{-1}}\mid L^{\infty}(\widetilde{\Omega})\|
=\frac{\|J_{\varphi_{A}^{-1}}\mid L^{\infty}(\widetilde{\Omega})\|}{\lambda_1(\widetilde{\Omega})},
\end{equation}
where $J_{\varphi_{A}^{-1}}$ is  a Jacobian of the inverse mapping
to the $A$-quasiconformal mapping $\varphi_{A}:\Omega \to \widetilde{\Omega}$.
\end{theorem}

As an application of Theorem~\ref{Th-LE} we consider several examples.

\begin{example}
\label{example1}
The homeomorphism
\[
\varphi(z)= \frac{a}{a^2-b^2}z- \frac{b}{a^2-b^2} \overline{z}, \quad z=x+iy, \quad a>b\geq 0,
\]
is an $A$-quasiconformal and maps the right triangle with arbitrary angels
$$
\Omega= \left\{(x,y) \in \mathbb R^2: 0\leq x\leq a+b,\,\, 0\leq y\leq -\frac{a-b}{a+b}x+(a-b)\right\}
$$
onto the $45^{\circ}$ right triangle
$$
\widetilde{\Omega}= \left\{(u,v) \in \mathbb R^2: 0\leq u\leq 1,\,\, 0\leq v\leq 1-u\right\}.
$$

The mapping $\varphi$ satisfies the Beltrami equation with
\[
\mu(z)=\frac{\varphi_{\overline{z}}}{\varphi_{z}}=-\frac{b}{a}
\]
and the Jacobian $J(z,\varphi)=|\varphi_{z}|^2-|\varphi_{\overline{z}}|^2=1/(a^2-b^2)$.
It is easy to verify that $\mu$ induces, by formula \eqref{Matrix-F}, the matrix function $A(z)$ form
$$
A(z)=\begin{pmatrix} \frac{a+b}{a-b} & 0 \\ 0 &  \frac{a-b}{a+b} \end{pmatrix}.
$$
Given that $|J(w,\varphi^{-1})|=|J(z,\varphi)|^{-1}=a^2-b^2$ and $\lambda_1(\widetilde{\Omega})=5\pi^2$ (see, for example, \cite{GN13}). Then by Theorem~\ref{Th-LE} we have
$$
\lambda_1(A,\Omega) \geq
\frac{\lambda_1(\widetilde{\Omega})}{\|J_{\varphi^{-1}}\mid L^{\infty}(\widetilde{\Omega})\|} = \frac{5\pi^2}{a^2-b^2}.
$$
\end{example}

\begin{example}
The homeomorphism
\[
\varphi(z)= \frac{2 \cdot z^{\frac{3}{8}}}{\overline{z}^{\frac{1}{8}}}-1,\,\, \varphi(0)=-1, \quad z=x+iy,
\]
is an $A$-quasiconformal and maps the interior of the non-convex domain
$$
\Omega:=\left\{(\rho, \theta) \in \mathbb R^2:\rho=\cos^{4}\left(\frac{\theta}{2}\right), \quad - \pi \leq \theta \leq \pi\right\}
$$
onto the unit disc $\mathbb D$.
The mapping $\varphi$ satisfies the Beltrami equation with
\[
\mu(z)=\frac{\varphi_{\overline{z}}}{\varphi_{z}}=-\frac{1}{3}\frac{z}{\overline{z}}
\]
and the Jacobian
$$J(z,\varphi)=|\varphi_{z}|^2-|\varphi_{\overline{z}}|^2=\frac{1}{2\cdot |z|^{\frac{3}{2}}}.
$$
We see that $\mu$ induces, by formula \eqref{Matrix-F}, the matrix function $A(z)$ form
$$
A(z)=\begin{pmatrix} \frac{|3\overline{z}+z|^2}{8|\overline{z}|^2} & \frac{3}{4}\Imag \frac{z}{\overline{z}} \\ \frac{3}{4}\Imag \frac{z}{\overline{z}} & \frac{|3\overline{z}-z|^2}{8|\overline{z}|^2} \end{pmatrix}.
$$
Given that $|J(w,\varphi^{-1})|=|J(z,\varphi)|^{-1}=2\cdot |z|^{\frac{3}{2}}$ and
$\lambda_1(\mathbb D)=(j_{0,1})^2$. Then by Theorem~\ref{Th-LE} we have
$$
\lambda_1(A,\Omega) \geq
\frac{\lambda_1(\mathbb D)}{\|J_{\varphi^{-1}}\mid L^{\infty}(\mathbb D)\|} \geq \frac{(j_{0,1})^2}{2}.
$$
\end{example}

Taking into account Theorem~\ref{Th-LE} and the domain monotonicity property of the Dirichlet eigenvalues for the elliptic operator in divergence form $L_{A}f(z)$, we obtain the
following result for the special case $\|J_{\varphi_{A}^{-1}}\mid L^{\infty}(\widetilde{\Omega})\|<1$.
\begin{proposition}\label{main}
Let $\Omega$ be an $A$-quasiconformal $\infty$-regular domain about $\widetilde{\Omega}$. We assume that $\widetilde{\Omega} \supset \Omega$, then
\[
\lambda_1(A,\Omega)-\lambda_1(\widetilde{\Omega}) \geq \frac{1-\|J_{\varphi_{A}^{-1}}\mid L^{\infty}(\widetilde{\Omega})\|}{ \|J_{\varphi_{A}^{-1}}\mid L^{\infty}(\widetilde{\Omega})\|}\lambda_1(\widetilde{\Omega}).
\]

\end{proposition}

\begin{proof}
Since $\widetilde{\Omega} \supset \Omega$, we have $\lambda_1(A,\Omega)\geq\lambda_1(\widetilde{\Omega})$.
Taking into account the inequality~\eqref{Estim} in Theorem~\ref{Th-LE} and making elementary calculation, we get
\[
\lambda_1(A,\Omega)-\lambda_1(\widetilde{\Omega}) \geq \frac{1-\|J_{\varphi_{A}^{-1}}\mid L^{\infty}(\widetilde{\Omega})\|}{ \|J_{\varphi_{A}^{-1}}\mid L^{\infty}(\widetilde{\Omega})\|}\lambda_1(\widetilde{\Omega}).
\]
\end{proof}

\subsection{The Rayleigh-Faber-Krahn type inequality}
The theory of composition operators \cite{GPU2020} allows us to reduce the spectral problem for the divergence form elliptic operator \eqref{EllDivOper} defined in simply connected bounded domain $\Omega \subset \mathbb C$ to a weighted spectral problem for
the Laplace operator in a simply connected bounded domain $\widetilde{\Omega}\subset \mathbb C$.
By the chain rule applied to a function $f(z)=g \circ \varphi_{A}(z)$ \cite{GNR18}, we have
\[
-\textrm{div} [A(z) \nabla f(z)]=-\textrm{div} [A(z) \nabla g \circ \varphi_{A}(z)]=-h(w)\Delta g(w),
\]
where the weight $h(w)=|J(w,\varphi_A^{-1})|^{-1}$ is  a Jacobian of the inverse mapping
to the $A$-quasi\-conformal mapping $\varphi_{A}:\Omega \to \widetilde{\Omega}$.

From here we can point out that
\begin{multline*}
\lambda_1(A, \Omega)= \inf_{f \in W_0^{1,2}(\Omega,A) \setminus \{0\}}
\frac{\iint\limits_{\Omega} \left\langle A(z) \nabla f, \nabla f\right\rangle dxdy}{\iint\limits_{\Omega} |f|^2dxdy} \\
= \inf_{g\in W_0^{1,2}(\widetilde{\Omega},h,1) \setminus \{0\}}
\frac{\iint\limits_{\widetilde{\Omega}} |\nabla g|^2dudv}{\iint\limits_{\widetilde{\Omega}} |g|^2h(w)dudv}=\lambda_1(h,\widetilde{\Omega}).
\end{multline*}

Let $\varphi_{A}:\Omega \to \widetilde{\Omega}$ be $A$-quasiconformal mappings for which $|J(z,\varphi_{A})|=1$ for almost all $z \in \Omega$. In this case quasiconformal weights $h(w)=|J(w, \varphi_{A}^{-1})|=1$ for almost all $w \in \widetilde{\Omega}$.
Hence we get that $\lambda_1(A, \Omega)=\lambda_1(\widetilde{\Omega})$.

Using this equality and the classical Rayleigh-Faber-Krahn inequality we immediately obtain  a version of this inequality for elliptic operators in divergence form.
Namely:
\begin{theorem}
Let $\Omega \subset \mathbb C$ be a simply connected  bounded domain such that there exists a measure preserving A-quasiconformal mapping $\varphi_{A}:\Omega \to \widetilde{\Omega}$. Then
\[
\lambda_1(A, \Omega) \geq \lambda_1(\Omega^{\ast})=\frac{{j_{0,1}^2}}{R^2_{\ast}},
\]
Here $\Omega^{\ast}$ is the disc of the radius $R_{\ast}$ such that $|\Omega|=|\Omega^{\ast}|$. and $j_{0,1} \approx 2.4048$ is the first positive zero of the Bessel function $J_0$.
\end{theorem}
\begin{proof}
By assumptions, the mapping $\varphi_{A}:\Omega \to \widetilde{\Omega}$ is measure preserving A-quasiconformal, i.e., 
$|J(z,\varphi_{A})|=1$ almost everywhere in $\Omega$. From the equality $|J(z,\varphi_{A})|=|J(z,\varphi^{-1}_{A})|=1$
 for almost all $z \in \Omega$ and almost all $w=\varphi_{A}(z) \in \widetilde{\Omega}$ we obtain the equality
 $\lambda_1(A, \Omega)=\lambda_1(\widetilde{\Omega})$. Now using the classical Rayleigh-Faber-Krahn inequality,
 \[
 \lambda_1(\widetilde{\Omega}) \geq \lambda_1(\Omega^{\ast})
 \]
 we obtain
 \[
\lambda_1(A, \Omega) \geq \lambda_1(\Omega^{\ast}).
\]
The theorem is proved.
\end{proof}

A description of the class of measure preserving $A$-quasiconformal mappings and/or corresponding divergence form elliptic equations is an open problem. Let us give simple examples of such mappings.

\begin{example}
Let $\varphi(z)=(ax,\frac{1}{a}y)$, $z=x+iy$ and $a>1$. Then $J(z,\varphi)=1$, the quasiconformality coefficient $Q_{\varphi}=a^2$ and
$\mu_{\varphi}=\frac{a^2-1}{a^2+1}$. The matrix can be easily recontracted
\[
A(z)= \begin{pmatrix} \frac{1}{a^2} & 0 \\ 0 &  a^2 \end{pmatrix}.
\]
\end{example}

A little bit more complicated example.
\begin{example}
Let $\varphi(z)=(ax+by, \frac{1}{a}y)$, $z=x+iy$ and $a>1$. Then $J(z,\varphi)=1$, the quasiconformality coefficient $Q_{\varphi}=a^2$.
Calculation of $\mu_{\varphi}$ is more complicated. We use the Beltrami equation, i.e
$$
\varphi_z=\frac{1}{2}(a+\frac{1}{a})-i\frac{b}{2}, \quad
\varphi_{\overline{z}}=\frac{1}{2}(a-\frac{1}{a})+i\frac{b}{2}.
$$
Hence
$$
\mu_{\varphi}=\frac{(a^2-1)(a^2+1)-a^2b^2}{(a^2+1)^2+a^2b^2}+i\frac{2a^3b}{(a^2+1)^2+a^2b^2}.
$$
The matrix $A$ can be easily reconstracted by elementary calculations.
\end{example}

\begin{example}
Let $f \in L^{1}_{\infty}(\mathbb R)$. Then $\varphi(z)=(x+f(y),y)$, $z=x+iy$, is a quasiconformal mapping with $|J(z,\varphi)|=1$ (see, \cite{GPU2020}).
A basic calculation implies
$$
\varphi_z=1-i\frac{f'(y)}{2}, \quad \varphi_{\overline{z}}=i\frac{f'(y)}{2}.
$$
Hence
$$
\mu_{\varphi}=-\frac{(f'(y))^2}{4+(f'(y))^2}+i\frac{2f'(y)}{4+(f'(y))^2}.
$$
 The matrix can be easily recontracted
\[
A(z)= \begin{pmatrix} 1 & -f'(y) \\ -f'(y) &  1+(f'(y))^2 \end{pmatrix}.
\]
\end{example}

This algorithm can be used for more complicated examples.

\vskip 0.3cm

\section{Appendix}

Few remarks about measure preserving and quasi-preserving (bi-Lipschitz) quasiconformal mappings and$A$-quasiconformal mappings.

A quasiconformal mapping $\varphi$ is a solution of the corresponding Beltrami equation
\begin{equation} \label{MP1}
\varphi_{\overline{z}}(z)=\mu(z) \varphi_{z}(z).
\end{equation}
Because $J(z, \varphi)=|\varphi_{z}^2|(z)-|\varphi_{\overline{z}}^2|(z)$, the condition $J(z,\varphi)=1$ can be written as
\begin{equation}\label{MP2}
|\varphi_{z}|^2(z)(1-|\mu(z)|^2)=1.
\end{equation}

Recall that for $A$-quasiconformal mappings we have
$$
\mu(z)=\frac{a_{22}(z)-a_{11}(z)-2ia_{12}(z)}{\det(I+A(z))}.
$$
By elementary calculations it is easy to verify that any measure preserving quasiconformal mapping $\varphi:\Omega \to \widetilde{\Omega}$ is a bi-Lipschitz mapping in the following sense: $D(\varphi) \in L^1_{\infty}(\Omega)$ and
$D(\varphi^{-1}) \in L^1_{\infty}(\widetilde{\Omega})$ (or $D\varphi \in L^{\infty}(\Omega)$ and
$D\varphi^{-1} \in L^{\infty}(\widetilde{\Omega})$). An equivalent geometric description is the following. The mapping $\varphi$ and its inverse are locally Lipschitz homeomorphisms with uniformly bounded Lipschitz constants.

The condition \eqref{MP2} can be soften in the spirit of quasiconformality up to the inequality
$$\
\frac{1}{C}<J(z,\varphi)<C
$$
for some positive constant $C$. We call such homeomorphism as quasi-preserving measure homeomorphisms.

For quasiconformal quasi-preserving measure homeomorphisms this condition can be written as
\begin{equation}\label{MP3}
\frac{1}{C}<|\varphi_{z}^2|(z)(1-|\mu(z)|^2)<C.
\end{equation}

The class of quasiconformal quasi-preserving measure homeomorphisms coincide with the class of bi-Lipschitz homeomorphisms. The constant $C$ can be easily recalculated in terms of uniform Lipschitz constants for the corresponding homeomorphism and its inverse.

We do not have any geometric description of such homeomorphisms i.e to solution of systems (\eqref{MP1},\eqref{MP2})
or (\eqref{MP1},\eqref{MP3}). Let us look for some simplified cases.

Suppose the matrix $A$ is a diagonal matrix. Then the quasiconformality coefficient $\mu(z)$ is a real number.
If $a_{11}(z)>0$ then $0<\mu(z)<1$ and the condition \eqref{MP2} can be simplified:
\begin{equation} \label{MP2s}
|\varphi_{z}|^2(z)(1-\mu(z)^2)=1.
\end{equation}
By simple calculations
$$
\mu(z)=\frac{a_{11}(z)+a_{11}^{-1}(z)}{2+a_{11}(z)+a_{11}^{-1}(z)}.
$$

Let us give an example of such homeomorphisms.
\begin{example}
Suppose $\Omega:=(0,1)\times(0,1)$ and $\varphi(x,y)=a(x)+ib(y)$ where $a(x)$ and $b(y)$ belong to $C^1(\Omega)$.
We also suppose that
$$
\inf\limits_{x\in(0,1)}(da/dx)>\sup\limits_{y\in(0,1)}(db/dy)
$$
and $I_1:=\inf\limits_{y\in(0,1)}(db/dy)>0$.

Any such mapping is a quasi-preserving measure one, because
$$
I_1 \leq J(z,\varphi)\leq \left(\left|da/dx\right|_{C^1(\Omega)}\times \left|db/dy\right|_{C^1(\Omega)} \right)
$$
and quasiconformal one because
$$
Q\leq \frac{\sup\limits_{x\in(0,1)}(da/dx)}{\inf\limits_{y\in(0,1)}(db/dy)}.
$$

The corresponding matrix $A$ can be easily reconstructed
\begin{equation}
A(z)= \begin{pmatrix} \frac{1-\mu(z)}{1+\mu(z)} & 0 \\ 0 &  \frac{1+\mu(z)}{1-\mu(z)} \end{pmatrix},\,\,\, \text{a.e. in}\,\,\, \Omega,
\end{equation}
where for $z=x+iy$
$$
\mu(z)=\frac{da/dx(x) -db/dy(y)}{da/dx(x) +db/dy(y)}\,.
$$

\end{example}

\textbf{Acknowledgements.}

 The second author was supported by the Ministry of Science and Higher Education of Russia (agreement No. 075-02-2022-884).

\vskip 0.3cm

Department of Mathematics, Ben-Gurion University of the Negev, P.O.Box 653, Beer Sheva, 8410501, Israel

\emph{E-mail address:} \email{vladimir@math.bgu.ac.il} \\

 Regional Scientific and Educational Mathematical Center, Tomsk State University, 634050 Tomsk, Lenin Ave. 36, Russia
							
 \emph{E-mail address:} \email{vpchelintsev@vtomske.ru}   \\

\end{document}